\documentclass[amssymb,reqno,amsfonts,refcheck,verbatim,righttag]{amsart}
\usepackage{graphicx}
\usepackage{graphics,color}
\usepackage{amssymb,mathrsfs}
\usepackage{graphicx,color,tikz,caption,subcaption}
\usepackage[all]{xy}
\usepackage{mathtools}
\usepackage{enumerate}
\usepackage{amsmath}
\usepackage{bbm}
\usepackage{bm}
\usepackage[colorlinks, linkcolor=blue,citecolor=blue, anchorcolor=blue,urlcolor=blue,hypertexnames=false]{hyperref}

\numberwithin{equation}{section} 
\newtheorem{thm}{Theorem}[section]
\newtheorem{lemma}[thm]{Lemma}

\newtheorem{de}[thm]{Definition}
\newtheorem{rem}[thm]{Remark}

\newcommand{\R}{\mathbb{R}}

\newcommand{\N}{\mathbb{N}}

\newcommand{\T}{\mathbb{T}}

\newcommand{\A}{\mathcal{A}}

\def\d{ \,\mathrm{d}}

\usepackage{float}
\begin{document}
\title{Quantitative Matrix-Driven Diophantine approximation on $M_0$-sets}
\author{Bo Tan \and Qing-Long Zhou$^{\dag}$}
\address{$^{1}$ School  of  Mathematics  and  Statistics,
                Huazhong  University  of Science  and  Technology, 430074 Wuhan, PR China}
          \email{tanbo@hust.edu.cn}

\address{$^{2}$ School  of  Mathematics  and  Statistics, Wuhan University of Technology, 430070 Wuhan, PR China }
\email{zhouql@whut.edu.cn}
\thanks{$^{\dag}$ Corresponding author.}

\date{}

\begin{abstract}
Let $E\subset [0,1)^{d}$ be a set supporting a probability measure $\mu$ with   Fourier decay $|\widehat{\mu}({\bf{t}})|\ll (\log |{\bf{t}}|)^{-s}$ for some constant $s>d+1.$ Consider a sequence of expanding integral matrices $\mathcal{A}=(A_n)_{n\in\N}$ such that  the minimal singular values of $A_{n+1}A_{n}^{-1}$ are uniformly bounded below by $K>1$. We prove a quantitative Schmidt-type counting theorem under the following   constraints:
(1) the points of interest are restricted to $E$;
(2) the denominators of the ``shifted'' rational approximations are drawn exclusively from $\mathcal{A}$.
Our result extends the work of Pollington, Velani, Zafeiropoulos, and Zorin (2022) to the matrix setting, advancing the study of Diophantine approximation on fractals. Moreover, it strengthens  the equidistribution property of the sequence $(A_n{\bf x})_{n\in\N}$    for $\mu$-almost every ${\bf x}\in E.$ Applications   include the normality of vectors and shrinking target problems on fractal sets.

\end{abstract}

\maketitle
\section{Introduction}
This paper further develops quantitative theory of Diophantine approximation on the so-called $M_0$-sets in $\mathbb{R}^{d}$. Throughout, for a real vector ${\bf x}=(x_1,\ldots,x_d),$ the notation $|{\bf x}|_{\infty}$ refers to the maximum norm, defined as
$$|(x_1,\ldots,x_d)|_{\infty}=\max_{1\le i\le d}|x_i|.$$
As usual, the Fourier transform of a non-atomic probability measure $\mu$ is defined by
$$\widehat{\mu}({\bf t})
:=\int e^{-2\pi i\langle{\bf t},{\bf x}\rangle}{\rm d}\mu({\bf x}) \ \ \ \ ({\bf t}\in \mathbb{R}^{d}),$$
where the scalar product $\langle{\bf t},{\bf x}\rangle$ is defined by
$$\langle{\bf t},{\bf x}\rangle=t_1x_1+\cdots+t_dx_d.$$
\begin{de}
The set $E$ is called an $M_0$-set if it supports a non-atomic probability Borel measure $\mu$ (${\rm supp}(\mu) \subset E$) whose Fourier transform $\widehat{\mu}$  vanishes at infinity, i.e.,
$$|\widehat{\mu}({\bf t})|\to 0 \ \text{as } |{\bf t}|_{\infty}\to \infty.$$
Such a measure $\mu$ is called a Rajchman measure.
\end{de}
The study of $M_0$-sets is motivated by two key aspects. First, numerous classical (including fractal) sets of interest are known to support Rajchman measures (see, for instance \cite{K80,QR03}), we refer to \cite{ABS22,ARS22, BA25,DGW24,HS15,LS20,P22} for some recent works on the logarithmic and polynomial Fourier decay for fractal measures. Second, there is a natural heuristic expectation that the behavior of such measures shares significant similarities with that of the Lebesgue measure. We focus on
the latter, especially in equidistribution and Diophantine approximation.

\subsection{Equidistribution on $M_0$-sets in $\mathbb{R}$}
A sequence $(x_n)_{n\in\N}$ is called \emph{equidistribution} (or \emph{uniformly distributed modulo one}) if for each sub-interval $[a,b]\subseteq[0,1],$ the following limit holds$\colon$
\begin{equation}\label{equi}
\lim_{N\to\infty}\frac{\sharp\big\{1\le n\le N\colon x_n \in [a,b]\!\!\!\!\pmod1\big\}}{N}=b-a,
\end{equation}
where $\sharp$ denotes the cardinality of a finite set.
The concept of equidistribution, first introduced in Weyl's groundbreaking 1916 paper \cite{Weyl16}, has since become a cornerstone of mathematical analysis, with deep connections to number theory, fractal geometry, and probability theory. Despite its elementary definition, proving equidistribution for explicit sequences often requires sophisticated techniques and remains a substantial challenge.
A notable example is the long-standing open problem of whether the sequence $\{(3/2)^n\}$
is equidistributed  \cite{D09,FLP95}. This question, closely tied to Diophantine approximation and the distribution of powers of $3/2$, has resisted resolution even with modern tools, including Weyl's criterion \cite{B12,KN74}, which reduces the problem to estimating associated exponential sums.

A pivotal result in uniform distribution theory, established by Davenport \emph{et al.} \cite{DEL63}, reveals a deep relationship between the generic distribution properties of the sequence $(q_nx)_{n\in\N}$ (where $x$ lies in the support of $\mu$) and  the decay rate of the Fourier transform $\widehat{\mu}.$

\begin{thm}[Davenport-Erd\H{o}s-Leveque, \cite{DEL63}]\label{DEL}
Let $\mu$ be a Borel probability measure supported on a subset $E\subseteq[0,1].$ Let $\mathcal{A}=(q_n)_{n\in\N}$
be a sequence of reals satisfying
$$\sum_{N=1}^{\infty}\frac{1}{N^{3}}\sum_{m,n=1}^{N}\widehat{\mu}(k(q_n-q_m))<\infty$$for any $k\in \mathbb{Z}\setminus\{0\},$ then for $\mu$ almost all $x\in E$ the sequence $(q_nx)_{n\in \N}$ is equdistributed.
\end{thm}

Recall that a sequence $\mathcal{A}=(q_n)_{n\in\N}$ is said to be \emph{lacunary} if there exists  $K>1$ such that $\mathcal{A}$ satisfies  the classical Hadamard gap condition
\begin{equation*}
\frac{q_{n+1}}{q_n}\ge K \ \ (n\in \N).
\end{equation*}
Theorem \ref{DEL} demonstrates that for any  probability measure $\mu$ on $E$ whose Fourier transform satisfies a polylogarithmic decay condition, the following holds$\colon$ for any lacunary sequence $(q_n)_{n\in\N}$ of positive reals,
the sequence $(q_nx)_{n\in\N}$ is equidistributed for $\mu$-almost all $x.$
In other words, given $y\in[0,1]$ and $r\in(0,1/2),$ for $\mu$-almost all $x\in E,$
\begin{equation*}
\lim_{N\to\infty}\frac{\sharp\big\{1\le n\le N\colon |\!|q_nx-y|\!|\le r\big\}}{N}=2r,
\end{equation*}
where $|\!|\alpha|\!|:=\min\{|\alpha-m|\colon m\in \mathbb{Z}\}$ denotes the distance from $\alpha\in \R$ to the nearest integer.  With this in mind, Pollington \emph{et al.}
\cite{PVZZ22} considered the situation in which the constant $r$ is replaced by a general positive function $\psi\colon \N\to (0,1)$, and established the significant quantitative result on the counting function
$$R(x,N; \gamma, \psi, \mathcal{A}):=\sharp\{1\le n\le N\colon |\!|q_nx-y|\!|\le \psi(q_n)\}.$$
\begin{thm}[Pollington-Velani-Zafeiropoulos-Zorin, \cite{PVZZ22}]\label{PVZZ}
Let $\mathcal{A}=(q_n)_{n\in\N}$ be a lacunary sequence of natural numbers.
Let $y\in [0,1]$ and $\psi\colon \N\to (0,1)$.
 Let $\mu$ be a probability measure supported on a subset $E$ of $[0,1]$. Suppose that
 there exists   $A>2$ such that
$$\widehat{\mu}(t)=O\left((\log |t|)^{-A}\right) ~~\text{as }  |t|\to \infty.$$
Then for any $\varepsilon>0,$ we have that
$$R(x,N; \gamma, \psi, \mathcal{A})=2\Psi(N)+O\Big(\Psi(N)^{2/3}(\log(\Psi(N)+2))^{2+\varepsilon}\Big)$$
for $\mu$-almost all $x\in E,$ where $\Psi(N)=\sum_{n=1}^{N}\psi(q_n).$
\end{thm}
\begin{rem}
The study of quantitative counting functions is driven by two fundamental questions in modern number theory:

\begin{itemize}
	\item Existence of normal numbers in fractal sets --– Understanding the prevalence of normal numbers within sets of fractional dimension.
	\item
	Metric Diophantine approximation on manifolds --– Establishing measure-theoretic laws for rational approximations to points lying on smooth  submanifolds.
\end{itemize}

Theorem \ref{PVZZ} implies that if a certain probability measure $\mu$ supported on  $E$ has sufficiently rapid Fourier decay, the set $E$ must contain normal numbers. An analogue of Khintchine-type Theorem may also be established.
\end{rem}
\begin{rem}The authors \cite{TZ24} extended the study to a  sequence $\mathcal{A}$ of real numbers (rather than  natural numbers)  and improved the associated error term in Theorem \ref{PVZZ}.
\end{rem}

We are naturally led to ask whether these conclusions remain valid in higher-dimensional   or matrix settings.\footnote{It is worth mentioning that a similar problem was proposed  by Baker-Khalil-Sahlsten \cite[Corollary 1.21]{BKS25}.} Our main result answers this question affirmatively.

\subsection{Equidistribution of vectors on $M_0$-sets}
The concept  of \emph{equidistribution} (or \emph{uniformly distributed modulo one})
for real sequences extends in a natural way to sequences in $\mathbb{R}^{d}$: simply replace the intervals $[a,b]$ in (\ref{equi}) with $d$-dimensional parallelepipeds. Following standard notation, we  write $e(x)=\text{exp}(2\pi ix).$

The sequence of $({\bf x}_n)_{n\in\N}=((x_{n,1},\ldots,x_{n,d}))_{n\in\N}$ of vectors of $\mathbb{R}^{d}$ is said to be \emph{equidistribution} if, for every $d$-dimensional parallelepiped $\mathcal{R}_{{\bf a},{\bf b}}:=[a_1,b_1]\times\cdots\times[a_d,b_d],$ the following limit holds$\colon$
\begin{equation*}
\lim_{N\to\infty}\frac{\sharp\big\{1\le n\le N\colon {\bf x}_n\in \mathcal{R}_{{\bf a},{\bf b}}\!\!\!\!\pmod1\big\}}{N}=\prod_{i=1}^{d}(b_i-a_i).
\end{equation*}
Weyl's equidistribution criterion in the torus $\mathbb{R}^{d}/\mathbb{Z}^{d}$ (see \cite[Theorem 1.17]{B12})
says that $({\bf x}_n)_{n\in\N}\in \mathbb{R}^{d}/\mathbb{Z}^{d}$ is equidistributed if and only if
\begin{equation}\label{DEL}
\lim_{N\to\infty}\frac{\sum_{n=1}^{N} e(-\langle{\bf k},{\bf x}_{n}\rangle)}{N}=0 \ \ {\text{for any } } {\bf k}\in \mathbb{Z}^{d}\setminus \{\bf 0\}.
\end{equation}

The following results utilize the singular values of a $d\times d$ matrix $A.$
Throughout,  let
$$\sigma(A):=\inf\big\{|\!|Ax|\!|_{2}\colon |\!|x|\!|_2=1\big\}$$
be the smallest singular value of $A,$ which is the square root of the smallest eigenvalue of $A^TA$.  Recent work \cite{BKS25} by Baker-Khalil-Sahlsten adapts the methodology of  Queff\'{e}le-Ramar\'{e} \cite[Lemma 7.3]{QR03}, originally developed  for the one dimensional Davenport-Erd\H{o}s-LeVeque criterion,  to establish the   following result.\footnote{A proof of Theorem \ref{BKH} appears in unpublished work by Tuomas Sahlsten and Jonathan Fraser.}
\begin{thm}[Baker-Khalil-Sahlsten, \cite{BKS25}]\label{BKH}
Let $\mu$ be a probability measure on $\mathbb{R}^{d}$ with polylogarithmic Fourier decay. For any expanding integral matrix $A$ on $\mathbb{R}^{d}$, the orbit $(A^{n}{\bf x})_{n\ge 1}$ is equidistributed for $\mu$-almost every ${\bf x}.$
\end{thm}

 By considering a sequence of  matrices, we extend this result to broader dynamical systems settings.
\begin{thm}\label{TZ2}
Let $\mu$ be a probability measure on $\mathbb{R}^{d}$ with polylogarithmic Fourier decay. Let $(A_n)_{n\in\N}$ be a sequence of expanding integral matrices. Assume that the minimal singular values of $A_{n+1}A_{n}^{-1}$ are uniformly bounded below by
$K>1$, the orbit $(A_n{\bf x})_{n\ge 1}$ is equidistributed for $\mu$-almost every ${\bf x}.$
\end{thm}


Let $r_i\colon \N\to (0,\frac{1}{2})$ be a real positive function for $1\le i\le d$, and define
${\bf r}:\N\to \R^d$ as
${\bf r}:=(r_1,\ldots,r_d).$
For $n\in \N,$ let
\begin{equation}\label{Rn}
 \mathcal{R}_{n}=\mathcal{R}( {\bf r}(n)):=\{{\bf x}\in [0,1)^{d}\colon |\!|x_i|\!|\le r_i(n),\ 1\le i\le d\}.
\end{equation}
When ${\bf r}\equiv{\bf a}$ is a constant function,  the parallelepiped  $\mathcal{R}( {\bf r}(n))$ is independent of $n$, which will be also denoted by $\mathcal{R}( {\bf a})$.
For ${\bf y}=(y_1,\ldots,y_d)\in [0,1)^{d},$ let  ${\bf y}+ \mathcal{R}_{n}$ denote the ``shifted'' parallelepiped with center at ${\bf y}$.
 Theorem \ref{TZ2} implies that for $\mu$-almost all $x\in E$, the sequence $(A_n{\bf x})_{n\in \N}$ modulo one ``hits''   ${\bf y}+\mathcal{R}({\bf a})$ with
the ``expected'' number of times. In other words, for $\mu$-almost all ${\bf x}\in E,$
\begin{equation}\label{limit}
\lim_{N\to\infty}\frac{\sharp\big\{1\le n\le N\colon A_n{\bf x} \in {\bf y}+\mathcal{R}({\bf a})\!\!\!\!\pmod1\big\}}{N}=2^{d}a_1\cdots a_d.
\end{equation}
We investigate  the setting where the sizes of the parallelepipeds are allowed to shrink with time. With this in mind, we study the counting function
$$R({\bf x},N)=R({\bf x},N;{\bf y}, {\bf r},\mathcal{A}):=\sharp\{1\le n\le N\colon A_n{\bf x} \in {\bf y}+\mathcal{R}( {\bf r}(n))\!\!\!\!\pmod1\}$$
As indicated in the definition, we will typically write $R({\bf x},N;{\bf y}, {\bf r},\mathcal{A})$
simply as $R({\bf x},N)$, since the remaining three parameters are generally fixed and clear from context. The following result shows that if $\widehat{\mu}$ decays sufficiently rapidly, then for $\mu$-almost every ${\bf x}\in E,$ the sequence $(A_n{\bf x})_{n\in \N}$ modulo 1 ``hits'' the shrinking parallelepipeds ${\bf y}+\mathcal{R}_{n}$ with the ``expected'' number of times.

\begin{thm}\label{TZ}
Let $\mu$ be a probability measure supported on a subset $E$ of $[0,1)^{d}$.
Let $(A_n)_{n\in\N}$ be a sequence of expanding integral matrices. Assume that the minimal singular values of $A_{n+1}A_{n}^{-1}$ are uniformly bounded below by
$K>1$,  and
there exists $s>d+1$ such that
$$\widehat{\mu}({\bf t})=O\Big((\log |{\bf t}|_{\infty})^{-s}\Big) ~~\text{as}~  |{\bf t}|_{\infty} \to \infty.$$
Then for any $\varepsilon>0,$ we have that
$$R({\bf x},N)=\Psi(N)+O\Big(\Psi(N)^{\frac{d}{d+1}}(\log \Psi(N)+2)^{2+\varepsilon}\Big)$$
for $\mu$-almost all ${\bf x}\in E,$ where $\Psi(N)=\sum_{n=1}^{N}2^{d}r_1(n)\cdots r_{d}(n).$
\end{thm}
\begin{rem}
(1) We observe that when $(A_n)_{n\in\N}$ is a sequence of diagonal matrices, for instance \begin{equation*}
A_n=\begin{pmatrix}
2^{n} & 0 \\
0 &  3^{n}
\end{pmatrix}, \end{equation*}
the coordinate components of the vector $A_n{\bf x}$ decouple into independent one-dimensional terms. Consequently,  the counting function's analysis simplifies to  one-dimensional framework via Fourier decay, mirroring the behavior of the analogous problem on $[0,1).$  By contrast, for non-diagonal matrices  $A_n$,  the components of  $A_n{\bf x}$   exhibit nontrivial  interdependent, introducing dynamics fundamentally distinct from the  one-dimensional case. We need new tools to deal with the resulting complexity.

(2) By definition, $\Psi(N)=2^{d}a_1\cdots a_d$ when $r_i(n)\equiv a_i$ for $1\le i\le d,$ and Theorem \ref{TZ} trivially yields $(\ref{limit})$ 
 with an explicit error term.

(3) The exponent in the error term of Theorem \ref{TZ} is likely suboptimal, and we expect this bound to be improved in subsequent work.
\end{rem}

\subsection{Applications}
In this section, we apply our earlier results to investigate normality of vectors and shrinking target problems for matrix transformations of tori on fractal sets.
\subsubsection{Normality of vectors on fractals}
 A vector ${\bf x}$ is called normal with respect to square matrix $A$ if $A^{n}{\bf x}$ is uniformly distributed modulo 1. This concept was introduced by Schmidt \cite{S64}.
To the best of our knowledge, the study of the normality of vectors on fractals in higher dimensions remains largely unexplored in the current literature.
\begin{itemize}

\item Simmons-Weiss  \cite{SW19} showed that, under the open set condition, the natural self-similar measure on any self-similar fractal ensures that almost every point in the fractal is  generic, which means the orbit under the Gauss map\footnote{The Gauss map $T$, which induces the continued fraction expansion, is defined as $Tx=\frac1x \!\!\pmod 1$ and $T0=0$. The Gauss measure $\d\mu=\frac1{\log2}\frac1{1+x}\d x$ is the unique ergodic measure with maximal entropy.} of such a point is equidistributed with respect to the corresponding Gauss measure.\footnote{A sequence $(x_n)_{n\in\N}$ is equidistribution if and only if the sampling measures $\frac1N\sum_{n=0}^{N-1}\delta_{x_n}$ converge weakly to the Lebesgue measure. More generally, if the sampling measures converge to a measure $\mu$, we say that $(x_n)$ is equidistribution with respect to $\mu$. }
 Analogous results for matrix approximation are presented.

\item Let $J\subset \mathbb{R}^{d}$ be the attractor of a finite iterated function system consisting of $n(\ge 2)$ maps of the form $${\bf x}\to A^{-1}{\bf x}+{\bf t}_i (i=1,\ldots,n),$$
where $A$ is an expanding $d\times d$ integral matrix. Under an irrationality condition on the translation parts ${\bf t}_i$, Dayan-Ganguly-Weiss \cite{DGW24} proved that, for any Bernoulli measure on $J$, almost every ${\bf x}\in J$
has an equidistributed orbit $(A^{n}{\bf x})_{n\in\N}.$
\end{itemize}
In both settings, the results rely on measure rigidity theorems for specific random walks on homogeneous spaces.  Theorem \ref{TZ2} establishes  that for any probability measure 
with polylogarithmic Fourier decay,  almost every ${\bf x}$ 
has an equidistributed orbit $(A_n{\bf x})_{n\in \N}.$
Furthermore, the theorem highlights how Fourier decay serves as a novel analytical method, enabling progress for measures that evade traditional dynamical or geometrical techniques.

\subsubsection{Shrinking targets on fractals}
Shrinking target problems, such as those in the continued fractions system \cite{LWWX14} and the
$\beta$-dynamical system \cite{SW13}, have emerged as central research themes across fractal geometry, metric number theory, and dynamical systems. These problems, which lie at the intersection of measure-theoretic and geometric analysis, continue to attract significant attention for their deep connections to equidistribution, entropy, and fractal dimension. For further information, see \cite{LLVWZ25, LLVZ23}.

The general framework for shrinking target problems in $\mathbb{R}^{d}$ is the following$\colon$
Let $T\colon X\to X$ be a continuous map defined on some Borel set $X\subset \mathbb{R}^{d}.$ Given a point ${\bf y}\in X$ and a sequence  $(E_n)_{n\in\N}$ of Borel subsets of $\mathbb{R}^{d},$ we defined 
$$W({\bf y}, (E_n)):=\{{\bf x}\in X\colon T^{n}({\bf x})\in {\bf y}+E_n \text{ for infinitely many }n\in\N\}.$$
Often $(E_n)_{n\in \N}$ is taken to be a nested sequence of sets containing the origin. Classical shrinking target problems in dynamical systems are primarily concerned with establishing  quantitative properties, such as (Lebesgue) measure and Hausdorff dimension, of the dynamically defined set
  $W({\bf y}, (E_n)).$

Let  $X$ be equipped with a Borel probability measure $\mu$. We investigate the $\mu$ measure of $W({\bf y}, (E_n).$
When the convergence condition $\sum_{n=1}^{\infty}\mu(\{{\bf x}\colon T^{n}({\bf x})\in {\bf y}+E_n \})<\infty$
  holds, the convergent part of the Borel-Cantelli Lemma yields immediately that $\mu(W({\bf y}, (E_n))=0 $.
  The divergence case
 presents a deeper challenge. Under the hypothesis that $T$ exhibits sufficiently rapid mixing (e.g., exponential mixing or polynomial mixing with optimal decay rate) with respect to $\mu$, the dynamical Borel-Cantelli principle asserts that
 $\mu(W({\bf y}, (E_n))=1$.
This dichotomy has been substantiated in diverse dynamical contexts, see for instance \cite{P67, CK01}.

In   recent advancement of the shrinking target theory, Li \emph{et al.} \cite{LLVZ23} revisited the foundational framework established by Hill-Velani \cite{HV99}, where $T: \mathbb{T}^{d}\to \mathbb{T}^{d}$ denotes a toral endomorphism on the  $d$-dimensional torus $\mathbb{T}^{d}:=\mathbb{R}^{d}/\mathbb{Z}^{d}.$ For a large class of subsets $(E_n)_{n\in\N}$ which includes metric balls, rectangular regions and hyperboloidal neighborhoods in algebraic varieties, they established the Lebesgue measure and Hausdorff dimension of $W({\bf y}, (E_n))$.  They obtained the quantitative refinement   beyond the  zero-full law, which significantly extends the original dichotomy in  \cite{HV99}.

\newcommand{\bfx}{\mathbf{x}}
\newcommand{\bfy}{\mathbf{y}}

Motivated by recent advances in the metric theory of Diophantine approximation on fractal sets, we establish new results concerning the shrinking target problem in fractal dynamical systems. Let $\A = (A_n)_{n \in \N}$ be a sequence of expanding integral matrices. The central object of our study is the limsup set defined as
\begin{equation*}
W_{\A}(\bfy, (\mathcal{R}_n)) := \left\{ \bfx \in \T^d : A_n \bfx\in \bfy +\mathcal{R}_n \!\!\!\!\!\pmod{1} \text{ for infinitely many } n \in \N \right\},
\end{equation*}
where $\mathcal{R}_n$ is defined in (\ref{Rn}).

We aim to establish a Khintchine-type theorem analogue concerning the measure-theoretic properties of the intersection $W_{\mathcal{A}}\big({\bf y}, (\mathcal{R}_n)\big)\cap E$, where $E$ denotes a measurable subset of $\mathbb{T}^{d}.$  As an immediate consequence of Theorem \ref{TZ}, we obtain the following result.
 \begin{thm}\label{TZ3}
Let $\mu$ be a probability measure supported on a subset $E$ of $\mathbb{T}^{d}$.
 Assume that the minimal singular values of $A_{n+1}A_{n}^{-1}$ are uniformly bounded below by
$K>1$,  and
there exists $s>d+1$ such that
$$\widehat{\mu}({\bf t})=O\Big((\log |{\bf t}|_{\infty})^{-s}\Big) ~~\text{as}~  |{\bf t}|_{\infty} \to \infty.$$
Then \begin{equation*}
\mu\big(W_{\mathcal{A}}\big({\bf y}, (\mathcal{R}_n)\big)\cap E\big)=\begin{cases}
   0   & \text{if  $\sum_{n=1}^{\infty}r_1(n)\cdots r_{d}(n)<\infty$}, \\
    1  & \text{if  $\sum_{n=1}^{\infty}r_1(n)\cdots r_{d}(n)=\infty$}.
\end{cases}
\end{equation*}
\end{thm}
\begin{rem}
Specifically, consider $A_n=A^{n}$ for some $A$. Theorem \ref{TZ3} thereby generalizes the    Lebesgue measure analyses in \cite{LLVZ23} to encompass general fractal measure frameworks.
\end{rem}

\section{Proof of Theorem \ref{TZ2}}
Before presenting the proof of Theorem \ref{TZ2},  we cite an auxiliary lemma \cite[Lemma 1.8]{B12}, which can be regarded as a higher dimensional  Davenport-Erd\H{o}s-Leveque criterion.
\begin{lemma}\label{higher-DEL}
Let $E$ be a subset of $\mathbb{R}^{d}$ and $\mu$ a Borel probability measure on $E.$ Let $(f_n)_{n\in\N}$ be a sequence of bounded measurable functions defined on $E.$ If the series
\begin{equation}\label{condition}\sum_{N=1}^{\infty}\frac{1}{N}\int_{E}\left|\frac{1}{N}\sum_{n=1}^{N}f_n(\bf{x})\right|^{2}{\rm d}\mu(\bf{x})<\infty,
\end{equation}
then for $\mu$-almost all ${\bf{x}}\in E$,
$$\lim_{N\to\infty}\frac{\sum_{n=1}^{N}f_n(\bf{x})}{N}=0.$$
\end{lemma}

Let $f_n({\bf{x}})=e(-\langle{\bf k},A_n{\bf x}\rangle)$ for $n\in \N.$ Then by Weyl's equidistribution criterion (\ref{DEL}) and Lemma \ref{higher-DEL}, it remains to verify that (\ref{condition}) holds. Fix ${\bf k}\in\mathbb{Z}^{d}\setminus \{\bf 0\}$ and $N\in \N.$ We deduce that
\begin{equation}\label{check}
\begin{split}
\frac{1}{N}\int_{E}\left|\frac{1}{N}\sum_{n=1}^{N}f_n(\bf{x})\right|^{2}{\rm d}\mu({\bf{x}})
&=\frac{1}{N^{3}}\sum_{n,m=1}^{N}\int_{E}e\big(\big\langle{\bf k},(A_n-A_m){\bf x}\big\rangle\big){\rm d}\mu({\bf{x}})\\
&=\frac{1}{N^{2}}+\frac{2}{N^{3}}\sum_{1\le m< n\le N}\widehat{\mu}((A_n-A_m)^{T}({\bf k})).
\end{split}
\end{equation}
Since $|\!|A({\bf k})|\!|_{2}\ge \sigma(A)|\!|{\bf k}|\!|_{2}\ge \sigma(A)$, we have
\begin{align*}
|\!|(A_n-A_m)^{T}({\bf k})|\!|_{2}&=|\!|A_m^{T}(A_nA_m^{-1}-I)^{T}({\bf k})|\!|_{2}
\\&\ge \sigma(A_m)|\!|(A_nA_m^{-1}-I)^{T}({\bf k})|\!|_{2}\\
&\ge  \sigma(A_m)\sigma(A_nA_m^{-1}-I)\ge K^{m-1}(K^{n-m}-1),
\end{align*}
where the last inequality follows by the assumptions of Theorem \ref{TZ2}. Since $K>1,$ there exists $m_0\in\N$ such that, whenever $n>m\ge m_0,$ we get
$$K^{m-1}\ge 1+\frac{1}{K-1}\ge 1+\frac{1}{K^{n-m}-1}=\frac{K^{n-m}}{K^{n-m}-1},$$
and thus
$$|\!|(A_n-A_m)^{T}({\bf k})|\!|_{2}\ge K^{n-m}.$$
By the Fourier decay $|\widehat{\mu}({\bf t})|=O((\log|{\bf t}|_{\infty})^{-\alpha})$, we have
\begin{align*}
 \text{r.h.s of } (\ref{check})&\le \frac{1}{N^{2}}+\frac{2m_0}{N^{2}}+\frac{2}{N^{3}}\sum_{n=1}^{N}\sum_{m=m_0}^{n-1}\widehat{\mu}((A_n-A_m)^{T}({\bf k}))\\
 &\ll \frac{2m_0+1}{N^{2}}+\frac{2}{N^{3}}\sum_{n=1}^{N}\sum_{m=m_0}^{n-1}(\log|\!|(A_n-A_m)^{T}({\bf k})|\!|_{2})^{-\alpha}\\
 &\le \frac{2m_0+1}{N^{2}}+\frac{2(\log K)^{-\alpha}}{N^{3}}\sum_{n=1}^{N}\sum_{m=m_0}^{n-1}\frac{1}{(n-m)^{\alpha}}\\
 &=O(\frac{1}{N^{2}}+\frac{1}{N^{1+\alpha}}).
\end{align*}
Since $\alpha>0,$ the series
$$\sum_{N=1}^{\infty}\frac{1}{N}\int_{E}\left|\frac{1}{N}\sum_{n=1}^{N}f_n(\bf{x})\right|^{2}{\rm d}\mu({\bf{x}})<\infty,$$
which completes the proof.

\section{Proof of Theorem \ref{TZ}}
Given ${\bf y}\in[0,1)^{d},$ and an  expanding matrix $A$.  Let
\begin{equation}\label{Eny}
	E_{n}^{\bf y}=E_{n}^{\bf y}({\bf r}):=\{{\bf x}\in [0,1)^{d}\colon A_{n}{\bf x}\in {\bf y}+\mathcal{R}({\bf r}(n))\!\!\!\!\pmod 1 \}.
\end{equation}

Then we have
$$R({\bf x},N):=\sharp\{n\in[1,N]\colon {\bf x}\in E_{n}^{\bf y}\}.$$
Thus, the sets $E_{n}^{\bf y}$ serve as the fundamental building blocks  of the structures studied in this paper. Let $\mu$ be a probability measure supported on a subset $E\subset [0,1)^{d}$ with $|\widehat{\mu}({\bf{t}})|\ll (\log |{\bf{t}}|_{\infty})^{-s}$ for  $s>d+1.$ We now proceed to estimate the $\mu$-measure of these building blocks.

The core idea originates  from Pollington-Velani \cite{PV00}, which has since been developed by several authors \cite{PVZZ22, PZ24} in the setting of $\mathbb{R}$. Due to $q^{-1}$-periodicity of the arcs manifests, their proof rely crucially on the property of $q$-divisibility. A natural question arises: does an analogous property holds in the more general Fourier space  $\mathbb{Z}^{d}$. We resolve the challenges arising from arbitrary (non-diagonal) matrices $A$ and establish the $A$-divisibility in $\mathbb{Z}^{d}$, a key ingredient in our proof.

\subsection{$A$-divisibility}

Let $\Gamma\subset\mathbb{R}^{d}$  be a overlattice of $\mathbb{Z}^d$ (i.e., $\Gamma$ is a discrete additive subgroup satisfying $\Gamma\supset \mathbb{Z}^d$).
To $\Gamma$, we associate the dual lattice
$$\Gamma^{\ast}=\{{\bf k}\in \mathbb{Z}^{d}\colon \langle {\bf k},{\bf x} \rangle\in \mathbb{Z}  \text{ for any } {\bf x}\in \Gamma\}.$$

The (higher dimensional) exponential sum   over $\Gamma$ is defined as$\colon$ for ${\bf k}\in \mathbb{Z}^{d},$
$$S({\bf k}):=\sum_{{\bf x}\in \Gamma/\mathbb{Z}^d} e(-\langle {\bf k},{\bf x} \rangle).$$
Here, since ${\bf k}\in \mathbb{Z}^{d}$,
the exponential term $e(-\langle {\bf k},{\bf x} \rangle)$  for
${\bf x}\in \Gamma/\mathbb{Z}^d$ is well-defined because its value is independent of the choice of representatives in the quotient $\Gamma/\mathbb{Z}^d$.

\begin{lemma}\label{dic}
The exponential sum $S({\bf k})$ satisfies the following dichotomy$\colon$
\begin{equation*}
S({\bf k})=\begin{cases}
  [\Gamma:\mathbb{Z}^d]  & \text{if } {\bf k}\in \Gamma^{\ast}, \\
  \mbox{}\quad  0   & \text{otherwise},
\end{cases}
\end{equation*}
where  $[\Gamma: \mathbb{Z}^d]$ is the index of $\mathbb{Z}^d$ in $\Gamma$
(as a subgroup of $\Gamma$).
\end{lemma}

\begin{proof}
By the definition of the dual lattice, if ${\bf k}\in \Gamma^{\ast}$ and $\bfx\in\Gamma$,  we have $e(-\langle {\bf k},{\bf x} \rangle)=1$, and thus $S({\bf k})=[\Gamma:\mathbb{Z}^d].$

If ${\bf k}\notin \Gamma^{\ast},$ there exists a element ${\bf x}_0\in\Gamma$ such that $\langle {\bf k},{\bf x}_0 \rangle \notin \mathbb{Z},$ which yields that $e(-\langle {\bf k},{\bf x}_0 \rangle)\neq 1.$
Choose a complete system of representatives $\{{\bf x}_1,\bfx_2, \ldots,{\bf x}_N\}$ $(N= [\Gamma:\mathbb{Z}^d]   )$ for the quotient group $\Gamma /\mathbb{Z}^d$. It follows that $S({\bf k})=\sum_{i=1}^N e(-\langle {\bf k},{\bf x}_i \rangle)$. We readily check that  $\{\bfx_0+{\bf x}_1,\bfx_0+\bfx_2, \ldots,\bfx_0+{\bf x}_N\}$ is also a complete system of representatives for the quotient, and thus
$$S({\bf k})=\sum_{i=1}^N e(-\langle {\bf k},\bfx_0+{\bf x}_i \rangle)=e(-\langle {\bf k},\bfx_0 \rangle)\sum_{i=1}^N e(-\langle {\bf k}, {\bf x}_i \rangle)=e(-\langle {\bf k},\bfx_0 \rangle)S({\bf k}).$$
Since $e(-\langle {\bf k},{\bf x}_0 \rangle)\neq 1$, we deduce that $S(\bf k)=0$, as desired.
\end{proof}

As a direct corollary of Lemma \ref{dic}, we establish
$A$-invariance on the $d$-dimensional torus $\mathbb{R}^{d}/\mathbb{Z}^{d}$, which is analogous to $q$-invariance on the unit circle.

\begin{lemma}\label{A-inv}
Let $A$ be a nonsingular integral matrix. Write
$$\mathcal{P}=\{{\bf p}\in\mathbb{Z}^{d} \colon {\bf p}=A{\bf x} \text{ for some } {\bf x}\in [0,1)^{d}\}.$$
 Then we have
 \begin{equation*}
\sum_{p\in\mathcal{P}}e(-\langle {\bf k},A^{-1}{\bf p} \rangle)=\begin{cases}
 |\det A|   & \text{if } (A^{T})^{-1}{\bf k}\in \mathbb{Z}^{d}, \\
   ~~ 0    & \text{otherwise}.
\end{cases}
\end{equation*}
\end{lemma}

\begin{proof}
Take $\Gamma=A^{-1}\mathbb{Z}^{d}.$ Since $A$ is a nonsingular integral matrix, we readily check that $\Gamma$ is an overlattice of $\mathbb{Z}^d$, and  its dual lattice  is
\begin{align*}
\Gamma^{\ast}&=\{{\bf k}\in\mathbb{Z}^{d}\colon \langle {\bf k},A^{-1}{\bf p} \rangle\in \mathbb{Z} \text{ for any }   {\bf p}\in \mathbb{Z}^{d}\}\\
&=\{{\bf k}\in\mathbb{Z}^{d}\colon \langle (A^{T})^{-1}{\bf k},{\bf p} \rangle\in \mathbb{Z} \text{ for any } {\bf p}\in \mathbb{Z}^{d}\}.
\end{align*}
It follows that  ${\bf k}\in \Gamma^{\ast}$ if and only if $(A^{T})^{-1}{\bf k}\in\mathbb{Z}^{d}.$

Moreover,  when ${\bf p}$
 runs over $\mathcal{P}$, the elements $A^{-1}{\bf p}$ form a complete set of coset representatives for $\Gamma/\mathbb{Z}^d$. Thus it  is clear that
$$\sharp \mathcal{P}=[\Gamma:\mathbb{Z}^d]=|\det A|.$$
Therefore, by Lemma \ref{dic}, we deduce that
 \begin{equation*}
\sum_{p\in\mathcal{P}}e(-\langle {\bf k},A^{-1}{\bf p} \rangle)
=\sum_{{\bf x}\in \Gamma/\mathbb{Z}^{d}}e(-\langle {\bf k},{\bf x} \rangle)=\begin{cases}
   |\det A|  & \text{if } {\bf k}\in \Gamma^{\ast}, \\
  ~~  0,  & \text{otherwise}.
\end{cases}
\end{equation*}
\end{proof}

\subsection{Warm-up}

The following result \cite[Theorem 3]{P67} (or \cite[Lemma~1.5]{H98}) constitutes a cornerstone of the metric theory of Diophantine approximation, providing a powerful framework for deriving precise counting estimates in asymptotic analyses. Originating from the classical variance method in probability theory, this result can be regarded as a quantitative refinement of the (divergence) Borel-Cantelli Lemma \cite[Lemma~2.2]{BRV16}.

\begin{lemma}\label{Philipp}
Let $(X,\mathcal{B},\mu)$ be a probability space, let $(f_n(x))_{n\in \N}$ be a sequence of non-negative $\mu$-measurable functions defined
on $X,$ and let $(f_n)_{n\in \N},$ $(\phi_n)_{n\in \N}$ be sequences of reals such that
$$0\le f_n\le \phi_n \ \ \ (n=1,2,\ldots).$$
Suppose that for arbitrary $a, b\in \N$ with $a<b,$ we have
$$\int_{X}\Big(\sum_{n=a}^{b}(f_n(x)-f_n)\Big)^{2}\d\mu(x)\le C\sum_{n=a}^{b}\phi_n$$
with  $C>0$ an absolute constant. Then, for any given $\varepsilon>0,$
$$\sum_{n=1}^{N}f_n(x)=\sum_{n=1}^{N}f_n+O\Big(\Psi(N)^{\frac{1}{2}}(\log\Psi(N))^{\frac{3}{2}+\varepsilon}+\max_{1\le n\le N}f_n\Big)$$
for $\mu$-almost all $x\in X,$ where $\Psi(N)=\sum_{n=1}^{N}\phi_n.$
\end{lemma}

Apply Lemma \ref{Philipp} with
$$X=[0,1)^{d}, f_n({\bf x})=\mathbb{I}_{E_{n}^{\bf y}}({\bf x}) \ \text{and} \  f_n=\psi(n):=2^{d}r_1(n)\cdots r_{d}(n),$$
where $\mathbb{I}_{E_{n}^{\bf y}}$ is the indicator function of the set $E_{n}^{\bf y}$ defined in (\ref{Eny}).
Hence
$$R({\bf x},N)=\sum_{n=1}^{N}f_{n}(x).$$
For $a, b\in \N$ with $a<b$, we calculate  that
\begin{align*}
&\Big(\sum_{n=a}^{b}(f_n({\bf x})-f_n)\Big)^{2}
\\&=\sum_{n=a}^{b}f_n({\bf x})+2\sum_{a\le m<n\le b}f_m({\bf x})f_n({\bf x})+\Big(\sum_{n=a}^{b}f_n\Big)^{2}-2\sum_{n=a}^{b}f_n\cdot\sum_{n=a}^{b}f_n({\bf x}),
\end{align*}
and thus
\begin{equation}\label{expect}
\begin{split}
&\int_X\Big(\sum_{n=a}^{b}(f_n({\bf x})-f_n)\Big)^{2}\d\mu({\bf x})
\\&=\sum_{n=a}^{b}\mu(E_{n}^{\bf y})
+2\sum_{a\le m<n\le b}\mu(E_{m}^{\bf y}\cap E_{n}^{\bf y})
-2\sum_{n=a}^{b}\psi(n)\Big(\sum_{n=a}^{b}\mu(E_{n}^{\bf y})-\frac{1}{2}\sum_{n=a}^{b}\psi(n)\Big).
\end{split}
\end{equation}

\subsection{Expectation transference principle}
In this subsection,  we first construct continuous     upper and lower bounding functions for the indicator function $\mathbb{I}_{E_{n}^{\bf y}}({\bf x})$, and then apply Fourier analysis to estimate the measure $\mu(E_{n}^{\bf y})$. Finally we derive the following expectation transference principle relating  the $\mu$-measures to Lebesgue measures:
$$\sum_{n=a}^{b}\mu(E_{n}^{\bf y})=\sum_{n=a}^{b}\mathcal{L}(E_{n}^{\bf y})+O(1).$$

For $n\in \mathbb{N}$, let $\varepsilon(n)\in(0,1]$ be a real number; we will specify its value later.
For $1\le i\le d,$ we define the upper step function  $\mathcal{X}_{r_i(n)}^{+}\colon [-1,1]\to \mathbb{R}$ by
\begin{equation*}
\mathcal{X}_{r_i(n)}^{+}(x)=
\begin{cases}
   1,  & \text{if } |x|\le r_i(n), \\
   1+\frac{1}{r_i(n)\varepsilon(n)}(r_i(n)-|x|), & \text{if } r_i(n)<|x|\le r_i(n)+ r_i(n)\varepsilon(n), \\
    0,  & \text{otherwise}.
\end{cases}
\end{equation*}
Correspondingly,  the lower step function is defined by
\begin{equation*}
\mathcal{X}_{r_i(n)}^{-}(x)=
\begin{cases}
   1,  & \text{if } |x|\le (1-\varepsilon(n))r_i(n), \\
   \frac{1}{r_i(n)\varepsilon(n)}(r_i(n)-|x|), & \text{if } (1-\varepsilon(n))r_i(n)<|x|\le r_i(n), \\
    0,  & \text{otherwise}.
\end{cases}
\end{equation*}
Putting $\mathcal{X}_{{\bf r}}^{+}({\bf x})=\prod_{i=1}^{d}\mathcal{X}_{r_i(n)}^{+}(x_i)$, we define the continuous upper approximation of
$\mathbb{I}_{E_{n}^{\bf y}}({\bf x})$ by
\begin{equation*}
g_{n}^{+}({\bf x}):=\Big(\sum_{{\bf p}\in \mathcal{P}_n}\delta_{A_n^{-1}({\bf p}+{\bf y})}({\bf x})\Big)\ast\mathcal{X}_{{\bf r}}^{+}({\bf x})=\sum_{{\bf p}\in \mathcal{P}_n}\mathcal{X}_{{\bf r}}^{+}(A_n{\bf x}-{\bf p}-{\bf y}),
\end{equation*}
where
$$\mathcal{P}_n=\{{\bf p}\in \mathbb{Z}^{d}\colon {\bf p}=A_n{\bf x} \text{ for some } {\bf x}\in [0,1)^{d}\},$$ and 
$\delta_{{\bf x}}$ denotes the Dirac delta-function at the point ${\bf x}.$  The lower approximation $g_n^-$ is defined analogously by replacing $\mathcal{X}_{r_i(n)}^{+}$ with $\mathcal{X}_{r_i(n)}^{-}$.
 It follows that
\begin{equation}\label{trans0}
\int_{X}g_{n}^{-}({\bf x}){\rm d}\mu({\bf x})\le \mu(E_{n}^{\bf y})\le \int_{X}g_{n}^{+}({\bf x}){\rm d}\mu({\bf x}).
\end{equation}
To evaluate the integrals, we express the integrands in terms of theirs Fourier series. For this purposes, we extend   $g_{n}^{+}({\bf x})$ and $g_{n}^{-}({\bf x})$ periodically to all of $\mathbb{R}^d$ be defining
$$h_{n}^{+}({\bf x})=\sum_{{\bf m}\in\mathbb{Z}^{d}}g_{n}^{+}({\bf x}+{\bf m}) \text{ and } h_{n}^{-}({\bf x})=\sum_{{\bf m}\in\mathbb{Z}^{d}}g_{n}^{-}({\bf x}+{\bf m}).$$

The following lemma calculates the Fourier coefficients of the functions $h_{n}^{\pm}$. 
In subsequent mathematical expressions, the $\pm$ symbol   indicates dual validity for   both the $+$ and $-$ cases (generally, multiple $\pm$ signs within a single formula must uniformly adopt either $+$ or $-$). Additionally the superscript notation ${}^*$   signifies taking the maximum between the ${}^+$ and ${}^-$ instances, for example, $|\widehat{h}_m^{*}\widehat{h}_n^{*} |$ is interpreted as $ \max\{|\widehat{h}_m^{+}\widehat{h}_n^{+}|, |\widehat{h}_m^{-}\widehat{h}_n^{-}|\}$.

\begin{lemma}\label{fc}
Let ${\bf k}=(k_1,\ldots,k_d)\in\mathbb{Z}^{d},$ we have$\colon$

(1) For  ${\bf k}={\bf 0},$
$\widehat{h}_n^{\pm}({\bf 0})=\psi(n)\cdot\big(1\pm\frac{\varepsilon(n)}{2}\big)^{d}. $ 

(2) For ${\bf k}\neq{\bf 0},$ 
\begin{itemize}
\item If $(A_n^{T})^{-1}{\bf k}\notin \mathbb{Z}^{d},$ $\widehat{h}_n^{\pm}({\bf k})=
    0.$

\item If $(A_n^{T})^{-1}{\bf k}\in \mathbb{Z}^{d},$  
$$ |\widehat{h}_n^{\pm}({\bf k})|\ll \prod_{i=1}^{d}\min\left\{r_{i}(n), \frac{1}{|k'_{i}|^{2}r_i(n)\varepsilon(n)}\right\}.$$
Here, $(A_n^{T})^{-1}{\bf k}={\bf k'} =(k_1',\ldots,k_d')\in \mathbb{Z}^{d}$, and we adopt the convention that $\frac{1}{0}=+\infty.$
\end{itemize}

\end{lemma}
\begin{proof}
We deduce
\begin{align*}
\widehat{h}_n^{+}({\bf k})&=\int_{X}h_n^{+}({\bf x})e(-\langle {\bf k},{\bf x} \rangle){\rm d}{\bf x}=\int_{\mathbb{R}^{d}}g_n({\bf x})e(-\langle {\bf k},{\bf x} \rangle){\rm d}{\bf x}\\
&=\sum_{{\bf p}\in \mathcal{P}_n}\int_{\mathbb{R}^{d}}\mathcal{X}_{{\bf r}}^{+}(A_n{\bf x}-{\bf p}-{\bf y})e(-\langle {\bf k},{\bf x} \rangle){\rm d}{\bf x}\\
&=\frac{1}{|{\rm det} A_n|}\sum_{{\bf p}\in \mathcal{P}_n}e(-\langle {\bf k},A_n^{-1}({\bf p} +{\bf y})\rangle)\int_{\mathbb{R}^{d}}\prod_{i=1}^{d}\mathcal{X}_{r_i(n)}^{+}(x_i)e(-\langle {\bf k},A_n^{-1}{\bf x} \rangle){\rm d}{\bf x}.
\end{align*}

(1)  If ${\bf k}={\bf 0},$ we have
\begin{align*}
\widehat{h}_n^{+}({\bf 0})=\frac{1}{|{\rm det} A_n|}\cdot \sharp \mathcal{P}_n\int_{\mathbb{R}^{d}}\prod_{i=1}^{d}\mathcal{X}_{r_i(n)}^{+}(x_i){\rm d}{\bf x}
=\psi(n)\cdot\big(1+\frac{\varepsilon(n)}{2}\big)^{d}.
\end{align*}

(2) Lemma \ref{A-inv} says that if $(A_n^{T})^{-1}{\bf k}\notin \mathbb{Z}^{d},$
$\sum_{{\bf p}\in \mathcal{P}_n}e(-\langle {\bf k},A_n^{-1}{\bf p} \rangle)=0$, and thus $\widehat{h}_n^{+}({\bf k})=0.$ On the other hand, if ${\bf k}=A_n^{T}{\bf k'}$, 
we have
\begin{align*}
\left|\widehat{h}_n^{+}(A_n^{T}{\bf k'})\right|&
=\left|e(-\langle {\bf k'},{\bf y}\rangle)\int_{\mathbb{R}^{d}}\prod_{i=1}^{d}\mathcal{X}_{r_i(n)}^{+}(x_i)e(-\langle {\bf k'},{\bf x} \rangle){\rm d}{\bf x}\right|\\
&\ll \int_{\mathbb{R}}\mathcal{X}_{r_{1}(n)}^{+}(x_{j+1}){\rm d}x_{j+1}\cdots\int_{\mathbb{R}}\mathcal{X}_{r_d(n)}^{+}(x_d){\rm d}x_{d}\\
&\ll \prod_{i=1}^{d}\min\left\{r_{i}(n), \frac{1}{|k'_{i}|^{2}r_i(n)\varepsilon(n)}\right\}.
\end{align*}

The same reasoning applies to the lower approximation, yielding analogous results for
 $\widehat{h}_n^{-}({\bf k}).$
\end{proof}

By Lemma \ref{fc}, we have
\begin{equation}\label{fss}
\begin{split}
 \sum_{{\bf k}\in \mathbb{Z}^{d}}|\widehat{h}_n^{\pm}({\bf k})|&\le2^{2d}+ \sum_{{\bf k'}\in \mathbb{Z}^{d}\setminus \{{0}\}} |\widehat{h}_n^{\pm}( A_n^{T}{\bf k'})|\\
&\ll 2^{2d}+\prod_{i=1}^{d}\sum_{|k'_i|=0}^{\infty}\min\left\{r_{i}(n), \frac{1}{|k'_{i}|^{2}r_i(n)\varepsilon(n)}\right\}\\
&\ll 2^{2d}+\prod_{i=1}^{d}\Big(\sum_{|k'_i|\ge \frac{1}{\sqrt{\varepsilon_n}r_i(n)}}\frac{1}{|k'_{i}|^{2}r_i(n)\varepsilon(n)}+\sum_{ |k'_i|< \frac{1}{\sqrt{\varepsilon_n}r_i(n)}}r_i(n)\Big)\\
&\asymp 2^{2d}+\varepsilon(n)^{-\frac{d}{2}}<+\infty,
\end{split}
\end{equation}
and thus the Fourier series $\sum_{{\bf k}\in \mathbb{Z}^{d}}\widehat{h}_n^{\pm}({\bf k})e(\langle {\bf k},{\bf x} \rangle)$ converge uniformly to $g_n^{\pm}({\bf x})$
respectively. It follows that
\begin{equation}\label{fs}
\int_{X}g_{n}^{\pm}({\bf x}){\rm d}\mu=\sum_{{\bf k}\in \mathbb{Z}^{d}}\widehat{h}_n^{\pm}({\bf k})\widehat{\mu}(-{\bf k}).
\end{equation}

\begin{lemma}\label{expect}
For   $a, b\in\mathbb{N}$ with $a<b,$ we have
$$\sum_{n=a}^{b}\mu(E_{n}^{\bf y})=\sum_{n=a}^{b}\psi(n)+O(1),$$
where the constant in $O(1)$  is  absolute, independent of $a, b$.
\end{lemma}
\begin{proof}
Combining (\ref{trans0}),  (\ref{fss}), (\ref{fs}) and Lemma \ref{fc}, we have
\begin{align*}
&\left|\sum_{n=a}^{b}\mu(E_{n}^{\bf y})-\sum_{n=a}^{b}\psi(n)\right|
\\
\ll & \sum_{n=a}^{b}\varepsilon(n)\psi(n)+ \sum_{n=a}^{b}\left|\sum_{{\bf k}\in \mathbb{Z}^{d}\setminus \{{0}\}}
\widehat{h}_n^{*}({\bf k}) \widehat{\mu}(-{\bf k})\right| \\
=&\sum_{n=a}^{b}\varepsilon(n)\psi(n)+ \sum_{n=a}^{b}\left|\sum_{{\bf k'}\in \mathbb{Z}^{d}\setminus \{{0}\}}
\widehat{h}_n^{*}(A_n^{T}{\bf k'}) \widehat{\mu}(-A_n^{T} {\bf k'})\right|\\
\le&\sum_{n=a}^{b}\varepsilon(n)\psi(n)+\sum_{n=a}^{b}n^{-s}\cdot \sum_{{\bf k'}\in \mathbb{Z}^{d}\setminus \{{0}\}} |\widehat{h}_n^{*}( A_n^{T}{\bf k'})|\\
\ll&\sum_{n=a}^{b}\varepsilon(n)\psi(n)+ \sum_{n=a}^{b}n^{-s}\varepsilon(n)^{-\frac{d}{2}}.
\end{align*}
Specially, for $\xi>0,$ we take
$$\varepsilon(n)=\min\left\{1,\Big(\sum_{k=1}^{n}\psi(k)\Big)^{-{1-\xi}}\right\}.$$
From $\psi(k)=\prod_{i=1}^{d}2r_i(k)\le 2^{d}$ and $s>1+d$, we deduce that
 $$\sum_{n=a}^{b}n^{-s}\varepsilon(n)^{-\frac{d}{2}}\ll \sum_{n=a}^{b}n^{-s+\frac{d}{2}}<\infty.$$
To handle the series $\sum_{n=a}^{b}\varepsilon(n)\psi(n)$, we employ a general convergence lemma  \cite[Lemma 3]{PZ24}$\colon$
Let $\{a_n\}\subset \mathbb{R}$ satisfying $a_1>0, a_n\ge 0$, and let
 $S_n=a_1+a_2+\cdots+a_n.$
Then for any $\xi>0$ the series
$$\sum_{n=1}^{\infty}\frac{a_n}{S_n^{1+\xi}}<\infty.$$
By setting  $a_n=\psi(n)$, we conclude that $\sum_{n=a}^{b}\varepsilon(n)\psi(n)<\infty$, which completes the proof.
\end{proof}

\begin{rem}\label{convergence} When $\sum_{n=1}^{\infty}\psi(n)<\infty,$  we
set $\varepsilon(n)=1$  in the proof of Lemma \ref{expect} to obtain
$ \sum_{n=a}^{b}\mu(E_{n}^{\bf y})\asymp \sum_{n=a}^{b}\psi(n)$
(here the condition $s>1$ suffices).  Applying  the convergence part of Borel-Cantelli Lemma yields that
$$\mu\Big(\limsup_{n\to\infty}E_{n}^{\bf y}\Big)=0.$$
\end{rem}

\subsection{Overlapping Measures}
This subsection focuses on estimating the sum of overlapping measures
$\sum_{a\le m<n\le b}\mu(E_{m}^{\bf y}\cap E_{n}^{\bf y}).$
To this end we employ Fourier-theoretic methods again.

%
Thanks to Remark \ref{convergence}, we can now restrict our attention to the key case where
$$\Psi(N)=\sum_{n=1}^{N}\psi(n)\to \infty \text{ as }N\to\infty.$$
 Without loss of generality, we may assume that the approximation functions satisfy a polynomial lower bound:
for any fixed $\tau > 1$, we have
\begin{equation}\label{assume}
r_i(n) \geq n^{-\tau} \quad \text{for all } n \in \mathbb{N} \text{ and } 1 \leq i \leq d.
\end{equation}
When this condition fails, we construct {regularized} functions $\widetilde{\mathbf{r}}$ where
\[\widetilde{ r}_i(n) = \max\{r_i(n), n^{-\tau}\}. \]

The counting functions satisfy the inequalities
\begin{align*}
R(\mathbf{x},N;\mathbf{y}, \mathbf{r}, \mathcal{A}) &\leq R(\mathbf{x},N;\mathbf{y}, \widetilde{\mathbf{r}}, \mathcal{A}) \\
&\leq R(\mathbf{x},N;\mathbf{y}, \mathbf{r}, \mathcal{A}) + \sum_{n=1}^N \mathbb{I}_{B(\mathbf{y}, n^{-\tau})}(\mathbf{x}_n).
\end{align*}
Since  $\sum_{n}n^{-\tau}$ converges, $\sum_{n=1}^N \mathbb{I}_{B(\mathbf{y}, n^{-\tau})}(\mathbf{x}_n)=O(1)$ for  $\mu$-almost all ${\bf x}\in E.$  The Diophantine approximation results for $\mathbf{r}$ and $\widetilde{\mathbf{r}}$ are equivalent. Consequently, we may proceed under the assumption (\ref{assume}).

Since
\begin{equation*}
\int_{X}g_{m}^{-}({\bf x})g_{n}^{-}({\bf x}){\rm d}\mu
\le \mu(E_{m}^{\bf y}\cap E_{n}^{\bf y})
\le \int_{X}g_{m}^{+}({\bf x})g_{n}^{+}({\bf x}){\rm d}\mu,
\end{equation*}
  it follows that
\begin{equation}\label{Overlap}
\begin{split}
&\quad |\mu(E_{m}^{\bf y}\cap E_{n}^{\bf y})-\widehat{h}_m^{\pm}( {\bf 0})\widehat{h}_n^{\pm}( {\bf 0})|\\
\le& \sum_{\substack{({\bf k}_1,{\bf k}_2) \in \mathbb{Z}^{d}\times\mathbb{Z}^{d}, \\ ({\bf k}_1,{\bf k}_2)\neq({\bf 0},{\bf 0})}} \left|\widehat{h}_m^{*}( {\bf k}_1)\widehat{h}_n^{*}( {\bf k}_2)\right|\left|\widehat{\mu}(-({\bf k}_1+{\bf k}_2))\right|\\
=&\sum_{{\bf k}_1^{'},{\bf k}_2^{'}\in \mathbb{Z}^{d}\setminus \{{\bf 0}\}} \left|\widehat{h}_m^{*}(A_m^{T} {\bf k}_1^{'})\widehat{h}_n^{*}( A_n^{T}{\bf k}_2^{'})\right|\left|\widehat{\mu}(-(A_m^{T}{\bf k}_1^{'}+A_n^{T}{\bf k}_2^{'}))\right|\\
=:&S(m,n).
\end{split}
\end{equation}

We now deal with the estimation of $S(m,n)$. To this end, we formulate an auxiliary lemma that will play a pivotal role in the estimation procedure. We begin with
 some notation. For a lattice $\Gamma$, define its non-negative and positive orthants respectively by
\begin{equation*}
    \begin{aligned}
        \Gamma_{\geq 0} &:= \Big\{ \mathbf{t} = (t_1, \dots, t_d) \in \Gamma:   \ t_i \geq 0 \text{ for } i = 1, \dots, d \Big\}, \\
        \Gamma_{> 0} &:= \Big\{ \mathbf{t} = (t_1, \dots, t_d) \in \Gamma:\ t_i > 0 \text{ for } i = 1, \dots, d \Big\}.
    \end{aligned}
\end{equation*}
A lattice $\Gamma$ is $\sigma$-discrete if the distance between any two distinct lattice points is at least $\sigma$.

\begin{lemma}\label{key}
Assume that a lattice $\Gamma$ in $\mathbb{R}^d$ is $\sigma$-discrete, where $\sigma>1$. For $r_1,\ldots,r_d,\varepsilon\in (0,1]$, we have
\begin{equation*}
\sum_{{\bf t}\in \Gamma\setminus\{{\bf 0}\}}\prod_{i=1}^{d}\min\Big\{r_i, \frac{1}{  t_i^{2}r_i\varepsilon}\Big\}\ll \sigma^{-1}\varepsilon^{-\frac{d}{2}},
\end{equation*}
  where the implicit constant depends solely on the dimension
$d$.
\end{lemma}

 \begin{proof}

We proceed by induction on the dimension \( d \).

\smallskip\noindent
\textbf{Base case (\( d = 1 \)):} A direct computation yields
\begin{equation*}
    \begin{split}
        &\sum_{t \in \Gamma, t > 0}  \min\left\{ r, \frac{1}{t^2 r \varepsilon} \right\}
        \leq \sum_{k=1}^\infty \min\left\{ r, \frac{1}{\sigma^2 k^2 r \varepsilon} \right\} \\
= &\sum_{k \leq (r \sigma \varepsilon^{1/2})^{-1}} r
        + \sum_{k > (r \sigma \varepsilon^{1/2})^{-1}} \frac{1}{\sigma^2 k^2 r \varepsilon} \ll \sigma^{-1} \varepsilon^{-1/2};
    \end{split}
\end{equation*}
an analogous bound holds for the summation over negative indices \( t < 0 \). Combining these bounds establishes the conclusion for the base case \( d = 1 \).

 \smallskip
\noindent
\textbf{Inductive step:} Decomposing the summation into \( 2^d \) coordinate orthants, it suffices to bound the first orthant summation:
\[
    \sum_{\bm{t} \in \Gamma_{\geq 0} \setminus \{\bm{0}\}} \prod_{i=1}^{d} \min\left\{ r_i, \frac{1}{t_i^2 r_i \varepsilon} \right\}.
\]
To this end, consider an auxiliary lattice \( \Delta \coloneqq \frac{\sigma}{\sqrt{d}} \mathbb{Z}^d \). This naturally induces a partition of the ambient space:
\[
    \mathbb{R}^d = \bigcup_{{\bf{k}} \in \Delta} \bigg( {\bf k} + \big[ 0, \frac{\sigma}{\sqrt{d}} \big)^d \bigg).
\]
Since \( \Gamma \) is \( \sigma \)-discrete, each cube \( {\bf k} + [0, \frac{\sigma}{\sqrt{d}})^d \) contains at most one lattice point of \( \Gamma \). Furthermore, if \( {\bf t} \in \Gamma_{\geq 0} \setminus \{\bm{0}\} \) resides in \( {\bf k} + [0, \frac{\sigma}{\sqrt{d}})^d \) with \( {\bf k} \in \Delta \), then  \( {\bf k} \in \Delta_{\geq 0} \setminus \{\bm{0}\} \), and  \( \min\left\{ r_i, \frac{1}{t_i^2 r_i \varepsilon} \right\} \leq \min\left\{ r_i, \frac{1}{k_i^2 r_i \varepsilon} \right\} \) for each \( i \).
This establishes the key inequality:
\begin{equation*}
    \sum_{{\bf t} \in \Gamma_{\geq 0} \setminus \{\bm{0}\}} \prod_{i=1}^{d} \min\left\{ r_i, \frac{1}{t_i^2 r_i \varepsilon} \right\}
    \leq \sum_{{\bf k} \in \Delta_{\geq 0} \setminus \{\bm{0}\}} \prod_{i=1}^{d} \min\left\{ r_i, \frac{1}{k_i^2 r_i \varepsilon} \right\}.
\end{equation*}

The summation over \( \Delta_{\geq 0} \setminus \{\bm{0}\} \) decomposes into:
\begin{itemize}
    \item The principal term: summation over \( \Delta_{> 0} \);

    \item Boundary terms:   \( d \) summations over \( (d-1) \)-dimensional lattices.
\end{itemize}
The boundary terms may be bounded via the induction hypothesis. The principal term over \( \Delta_{> 0} \) admits an explicit bound through direct computation analogous to the base case.
\end{proof}

We are ready to bound $S(m,n)$.

\begin{lemma}\label{intersection}
For $m<n,$ we have that 
\begin{align*}
S(m,n)\ll n^{-s}\varepsilon(m)^{-\frac{d}{2}}\varepsilon(n)^{-\frac{d}{2}}+\widehat{h}_m^{+}( {\bf 0}) \varepsilon(n)^{-\frac{d}{2}} n^{-s}+\widehat{h}_n^{+}( {\bf 0})\varepsilon(m)^{-\frac{d}{2}}m^{-s}+\frac{\varepsilon(m)^{-\frac{d}{2}}\psi(n)}{K^{n-m}}.
\end{align*}
\end{lemma}

\begin{proof}
Write $$S(m,n)=S_1(m,n)+S_2(m,n)+S_3(m,n),$$
where
\begin{align*}
S_1(m,n)&:=\sum_{{\bf k}_2^{'}\in \mathbb{Z}^{d}\setminus \{{\bf 0}\}} \left|\widehat{h}_m^{*}( {\bf 0})\widehat{h}_n^{*}( A_n^{T}{\bf k}_2^{'})\right|\left|\widehat{\mu}(-A_n^{T}{\bf k}_2^{'})\right|,\\
S_2(m,n)&:=\sum_{{\bf k}_1^{'}\in \mathbb{Z}^{d}\setminus \{{\bf 0}\}} \left|\widehat{h}_n^{*}( {\bf 0})\widehat{h}_m^{*}( A_n^{T}{\bf k}_1^{'})\right|\left|\widehat{\mu}(-A_m^{T}{\bf k}_1^{'})\right|,\\
S_3(m,n)&:=\sum_{{\bf k}_1^{'},{\bf k}_2^{'}\in \mathbb{Z}^{d}\setminus \{{\bf 0}\}} \left|\widehat{h}_m^{*}(A_m^{T} {\bf k}_1^{'})\widehat{h}_n^{*}( A_n^{T}{\bf k}_2^{'})\right|\left|\widehat{\mu}(-(A_m^{T}{\bf k}_1^{'}+A_n^{T}{\bf k}_2^{'}))\right|.
\end{align*}

We proceed case-by-case. From  (\ref{fss}) and $|\!|A_n^{T}{\bf k}_2^{'}|\!|_2\ge K^{n-1}$, we deduce that
$$S_1(m,n)\ll\widehat{h}_m^{+}( {\bf 0})\cdot \varepsilon(n)^{-\frac{d}{2}}\cdot n^{-s}.$$

Symmetrically, we obtain
$$S_2(m,n)\ll\widehat{h}_n^{+}( {\bf 0})\cdot \varepsilon(m)^{-\frac{d}{2}}\cdot m^{-s}.$$

We now decompose the set $S_3(m,n)$ into three parts based on a fixed index \( \alpha \in (0,1) \).

\noindent{\underline{{Case 1:}} \quad
$\Omega_1:=\big\{({\bf k}_1^{'}, {\bf k}_2^{'})\in \big(\mathbb{Z}^{d}\setminus \{{\bf 0}\}\big)^2: |\!|A_m^{T}{\bf k}_1^{'}+A_n^{T}{\bf k}_2^{'}|\!|_2\ge \frac{\sigma(A_n)}{2}\big\}.$}

Inequality (\ref{fss}) and $\sigma(A_n)\ge K^{n-1}$ yields that
\begin{align*}
&\sum_{\Omega_1}~\left|\widehat{h}_m^{*}(A_m^{T} {\bf k}_1^{'})\widehat{h}_n^{*}( A_n^{T}{\bf k}_2^{'})\right|\left|\widehat{\mu}(-(A_m^{T}{\bf k}_1^{'}+A_n^{T}{\bf k}_2^{'}))\right|\\
\ll& n^{-s}\sum_{{\bf k}_1^{'}\in \mathbb{Z}^{d}\setminus \{{\bf 0}\}}|\widehat{h}_m^{*}(A_m^{T} {\bf k}_1^{'})| \sum_{{\bf k}_2^{'}\in \mathbb{Z}^{d}\setminus \{{\bf 0}\}}|\widehat{h}_m^{*}(A_n^{T} {\bf k}_1^{'})|\\
\ll& n^{-s}\varepsilon(m)^{-\frac{d}{2}}\varepsilon(n)^{-\frac{d}{2}}.
\end{align*}

\noindent{\underline{{Case 2:}} \quad
$\Omega_2:=\big\{({\bf k}_1^{'}, {\bf k}_2^{'})\in \big(\mathbb{Z}^{d}\setminus \{{\bf 0}\}\big)^2: \frac{(\sigma(A_m))^{\alpha}}{2}\le|\!|A_m^{T}{\bf k}_1^{'}+A_n^{T}{\bf k}_2^{'}|\!|_2< \frac{\sigma(A_n)}{2}\big\}.$}

From
$$\sigma(A_n)|\!|(A_mA_n^{-1})^{T}{\bf k}_1^{'}+{\bf k}_2^{'}|\!|_2\le |\!|A_m^{T}{\bf k}_1^{'}+A_n^{T}{\bf k}_2^{'}|\!|_2< \frac{\sigma(A_n)}{2},$$
we deduce
\begin{equation}\label{res1}
|\!|(A_mA_n^{-1})^{T}{\bf k}_1^{'}+{\bf k}_2^{'}|\!|_2<\frac{1}{2}.
\end{equation}
Hence, to any non-zero integer vector ${\bf k}_1^{'}$, there corresponds at most one  non-zero integer vector ${\bf k}_2^{'}$ satisfying   (\ref{res1}).
Furthermore, combining  $|\!|A_m^{T}{\bf k}_1^{'}+A_n^{T}{\bf k}_2^{'}|\!|_2\ge\frac{(\sigma(A_m))^{\alpha}}{2}$ and
$\widehat{\mu}({\bf t})=O\Big((\log |{\bf t}|_{\infty})^{-s}\Big)$ yields that  $\widehat{\mu}(-(A_m^{T}{\bf k}_1^{'}+A_n^{T}{\bf k}_2^{'}))\ll m^{-s}$; Lemma \ref{fc}(2) implies that $\widehat{h}_n^{*}( A_n^{T}{\bf k}_2')
\ll \psi(n)$. Therefore we conclude that the   summation over this part satisfies
\begin{align*}
\sum_{\Omega_2}\ll m^{-s}\psi(n)\cdot \sum_{{\bf k}_1^{'}\in \mathbb{Z}^{d}\setminus \{{\bf 0}\}}\left|\widehat{h}_m^{*}(A_m^{T} {\bf k}_1^{'})\right|\ll m^{-s}\psi(n)\varepsilon(m)^{-\frac{d}{2}}.
\end{align*}

\noindent{\underline{{Case 3:}} \quad
$\Omega_3:=\big\{({\bf k}_1^{'}, {\bf k}_2^{'})\in \big(\mathbb{Z}^{d}\setminus \{{\bf 0}\}\big)^2: |\!|A_m^{T}{\bf k}_1^{'}+A_n^{T}{\bf k}_2^{'}|\!|_2<\frac{(\sigma(A_m))^{\alpha}}{2}\big\}.$}

Since
$$\sigma(A_m)|\!|{\bf k}_1^{'}+(A_nA_m^{-1})^{T}{\bf k}_2^{'}|\!|_2
\le|\!|A_m^{T}{\bf k}_1^{'}+A_n^{T}{\bf k}_2^{'}|\!|_2
<\frac{(\sigma(A_m))^{\alpha}}{2},$$
we have
\begin{equation}\label{res2}
|\!|{\bf k}_1^{'}+(A_nA_m^{-1})^{T}{\bf k}_2^{'}|\!|_2<\frac{1}{2}.
\end{equation}

The argument hinges on controlling the summation over $\Omega_3$ through three interconnected observations.

First, each $\mathbf{k}_2' \in \mathbb{Z}^d \setminus \{\mathbf{0}\}$ corresponds to at most one $\mathbf{k}_1'$ satisfying \eqref{res2}, and for $\mathbf{t}(\mathbf{k}_2') := (A_nA_m^{-1})^T\mathbf{k}_2' = (t_1, \ldots, t_d)$ with its associated $\mathbf{k}_1' = (k_1, \ldots, k_d)$ (if nonempty), the inequality $|t_i - k_i| \leq 1$ componentwise implies
\[
\min\left\{r_{i}(m), \frac{1}{|k_i|^{2}r_i(m)\varepsilon(m)}\right\} \asymp \min\left\{r_{i}(m), \frac{1}{|t_i|^{2}r_i(m)\varepsilon(m)}\right\}.
\]

Second, the lattice $\Gamma = (A_nA_m^{-1})^T\mathbb{Z}^d$ inherits sparsity from the singular value bound $\sigma(A_nA_m^{-1}) \geq K^{n-m}$, enforcing $\|\mathbf{t}_1 - \mathbf{t}_2\|_2 \geq K^{n-m}$ for all distinct $\mathbf{t}_1, \mathbf{t}_2 \in \Gamma$. In other words, the lattice $\Gamma$ is $K^{n-m}$-discrete.

 Third, combining the trivial Fourier coefficient bound $|\widehat{\mu}(\mathbf{t})| \leq 1$ with Lemma~\ref{fc}(2), we dominate
\[
\sum_{\Omega_3} \ll \psi(n) \sum_{\mathbf{k}_2' \in \mathbb{Z}^d \setminus \{\mathbf{0}\}} \prod_{i=1}^{d} \min\left\{r_{i}(m), \frac{1}{|k_i|^{2}r_i(m)\varepsilon(m)}\right\},
\]
which transfers via $\mathbf{k}_2' \mapsto \mathbf{t}$ equivalence to
\[
\psi(n) \sum_{\mathbf{t} \in \Gamma \setminus \{\mathbf{0}\}} \prod_{i=1}^{d} \min\left\{r_{i}(m), \frac{1}{|t_i|^{2}r_i(m)\varepsilon(m)}\right\}.
\]
Lemma~\ref{key} then applies to this product kernel over the $K^{n-m}$-separated lattice $\Gamma$, yielding the ultimate bound
\[
\sum_{\Omega_3} \ll \frac{\psi(n)\varepsilon(m)^{-d/2}}{K^{n-m}}.
\]

Synthesizing the bounds from all cases, we complete the proof.
\end{proof}

We now address the   estimate for the correlation sum  $\sum_{a\le m<n\le b}\mu(E_{m}^{\bf y}\cap E_{n}^{\bf y}).$
Define the scaling parameter
\begin{equation*}
    \varepsilon(n) := \min\left\{1, \left(\sum_{k=a}^{n} \psi(k)\right)^{-\delta}\right\}, \quad 0 < \delta \leq 1,
\end{equation*}
where $\delta$ serves as the critical parameter governing the decay rate, to be optimized later. Given  $\psi(k) \leq 2^{d}$, we derive that
\begin{equation}\label{ine1}
\varepsilon(n)^{-1}\le \max\{1, 2^{d}n\}<2^{d}n.
\end{equation}
And thus, for $s>1+d,$ we have that
\begin{equation}\label{ine}
\sum_{n=1}^{\infty}n^{-s}\varepsilon(n)^{-\frac{d}{2}}\ll \sum_{n=1}^{\infty}n^{-s}n^{\frac{d}{2}}<+\infty.
\end{equation}

\begin{lemma}\label{overlap}
We have
$$\sum_{a\le m<n\le b}\mu(E_{m}^{\bf y}\cap E_{n}^{\bf y})\le \frac{1}{2}\Big(\sum_{n=a}^{b}\psi(n)\Big)^{2}+O\Big(\Big(\sum_{n=a}^{b}\psi(n)\Big)^{\frac{2d}{d+1}}\log^{+}\Big(\sum_{n=a}^{b}\psi(n)\Big)\Big).$$
\end{lemma}
\begin{proof}
By Lemma~\ref{intersection} and the  inequality~\eqref{Overlap}, we obtain the upper bound decomposition:
\begin{align*}
\sum_{a \leq m < n \leq b} \mu(E_{m}^{\mathbf{y}} \cap E_{n}^{\mathbf{y}})
\ll & \sum_{a \leq m < n \leq b} \Big[ \widehat{h}_m^{+}(\mathbf{0})\widehat{h}_n^{+}(\mathbf{0})
+ n^{-s}\varepsilon(m)^{-\frac{d}{2}}\varepsilon(n)^{-\frac{d}{2}} \\
& + \widehat{h}_m^{+}(\mathbf{0}) \varepsilon(n)^{-\frac{d}{2}} n^{-s}
+ \widehat{h}_n^{+}(\mathbf{0})\varepsilon(m)^{-\frac{d}{2}}m^{-s} + \frac{\varepsilon(m)^{-\frac{d}{2}}\psi(n)}{K^{n-m}} \Big].
\end{align*}
The subsequent analysis bifurcates based on the dichotomy condition:
\[
\sum_{k=a}^{b} \psi(k) \leq 1 \quad \text{versus} \quad \sum_{k=a}^{b} \psi(k) > 1.
\]

\noindent
\textbf{Case 1: $\sum_{k=a}^{b}\psi(k) \leq 1$}.
Under this dichotomy, the scaling parameter simplifies to $\varepsilon(n) \equiv 1$ for all $n \in \mathbb{N}$ by definition. We analyze each term in the upper bound decomposition:

  By Lemma~\ref{fc}(1), we have that
    \begin{align*}
        \sum_{a \le m < n \le b} \widehat{h}_m^{+}(\mathbf{0})\widehat{h}_n^{+}(\mathbf{0})
         \ll\sum_{a\le m<n\le b}\psi(m)\psi(n)
\le \Big(\sum_{n=a}^{b}\psi(n)\Big)^{2}\ll \sum_{n=a}^{b}\psi(n).
    \end{align*}

 Through  \eqref{assume} with $\tau = \frac{d+s-1}{2d}$, we obtain
    \begin{align*}
        \sum_{a \le m < n \le b} n^{-s}
        &\leq \sum_{a \le m < n \le b} n^{d\tau - s} \psi(n)\ll \sum_{n=a}^{b} \psi(n) \sum_{m=a}^{n-1} m^{-(s - d\tau)}\ll \sum_{n=a}^{b}\psi(n).
    \end{align*}

Lemma~\ref{fc}(1) with \eqref{ine} yields
    \begin{equation*}
        \sum_{a \le m < n \le b} \left[\psi(m)n^{-s} + \psi(n)m^{-s}\right] \ll \sum_{n=a}^{b} \psi(n).
    \end{equation*}

The geometric series summation gives
    \begin{equation*}
        \sum_{a \le m < n \le b} \frac{\psi(n)}{K^{n-m}}
        \leq \sum_{n=a}^{b} \psi(n) \sum_{m=a}^{n-1} K^{m-n}
        \ll \sum_{n=a}^{b} \psi(n).
    \end{equation*}

\noindent
Synthesizing completes the estimate:
\begin{equation*}
    \sum_{a \le m < n \le b} \mu(E_{m}^{\mathbf{y}} \cap E_{n}^{\mathbf{y}}) = O\left(\sum_{n=a}^{b} \psi(n)\right).
\end{equation*}

\noindent
\textbf{Case 2: $\sum_{k=a}^{b}\psi(k) > 1$}.
Here the scaling parameter becomes nontrivial: $\varepsilon(n) = \left(\sum_{k=a}^{n}\psi(k)\right)^{-\delta} < 1$. We proceed with term-by-term analysis:

Lemma~\ref{fc}(1) with $\varepsilon(m)\varepsilon(n) < \varepsilon(m)$ gives
    \begin{align*}
        \widehat{h}_m^{+}(\mathbf{0})\widehat{h}_n^{+}(\mathbf{0})
        &= \psi(m)\psi(n)\left(1 + \frac{\varepsilon(m)}{2}\right)^d\left(1 + \frac{\varepsilon(n)}{2}\right)^d \\
        &\leq \psi(m)\psi(n) + 2^d\varepsilon(m)\psi(m)\psi(n).
    \end{align*}
    Summing over indices, we have
    \begin{equation*}
        \sum_{a \leq m < n \leq b} \widehat{h}_m^{+}(\mathbf{0})\widehat{h}_n^{+}(\mathbf{0})
        \leq \frac{1}{2}\left(\sum_{n=a}^b \psi(n)\right)^2 + 2^d\sum_{a \leq m < n \leq b} \varepsilon(m)\psi(m)\psi(n).
    \end{equation*}

    For the next term, we   cite a  harmonic series bound lemma  \cite[Lemma 4]{PZ24}$\colon$
Let $\{s_k\}$ be a sequence of nonnegative real numbers with $s_1\neq 0$. Write
$$S_n=s_1+\cdots s_n, \ \ n\in\N.$$
Then for any $N\in \N,$ the partial sums
$$\sum_{n=1}^{N}\frac{s_k}{S_n}\le 1+\log(S_N)-\log(s_1).$$

Applying the lemma with $s_k = \psi(a+k-1)$, we obtain
\begin{align*}
        \sum_{m=a}^b \varepsilon(m)\psi(m)
        &= \left(\sum_{n=a}^b \psi(n)\right)^{1-\delta} \sum_{k=1}^{b-a+1} \frac{s_k}{S_k} \\
        &\ll \left(\sum_{n=a}^b \psi(n)\right)^{1-\delta} \log\left(\sum_{n=a}^b \psi(n)\right).
    \end{align*}
Consequently,
\begin{equation*}
2^d \sum_{a \leq m < n \leq b} \varepsilon(m)\psi(m)\psi(n)
 \ll \left(\sum_{n=a}^b \psi(n)\right)^{2-\delta} \log\left(\sum_{n=a}^b \psi(n)\right).
\end{equation*}

Using \eqref{ine1} and $s > d+1$, we have that
    \begin{align*}
        \sum_{a \leq m < n \leq b} \left[\widehat{h}_m^{+}(\mathbf{0})\varepsilon(n)^{-d/2}n^{-s} + \widehat{h}_n^{+}(\mathbf{0})\varepsilon(m)^{-d/2}m^{-s}\right]
        &\ll \sum_{n=a}^b \psi(n).
    \end{align*}

     Geometric series control yields
    \begin{equation*}
        \sum_{a \leq m < n \leq b} \frac{\varepsilon(m)^{-d/2}\psi(n)}{K^{n-m}}
        \ll \left(\sum_{n=a}^b \psi(n)\right)^{1 + d\delta/2}.
    \end{equation*}

Synthesizing  and setting $\delta = \frac{2}{d+1}$ to balance exponents, we conclude that
\begin{align*}
    \sum_{a \leq m < n \leq b} \mu(E_m^{\mathbf{y}} \cap E_n^{\mathbf{y}})
    &\leq \frac{1}{2}\left(\sum_{n=a}^b \psi(n)\right)^2
    + O\Biggl(\left(\sum_{n=a}^b \psi(n)\right)^{\frac{2d}{d+1}} \log^+\left(\sum_{n=a}^b \psi(n)\right) \\
    &\quad + \left(\sum_{n=a}^b \psi(n)\right)^{\frac{2d}{d+1}} + \left(\sum_{n=a}^b \psi(n)\right)^{\frac{2d}{d+1}}\Biggr).
\end{align*}
\end{proof}

We now advance to the moment estimation. By (\ref{expect}) and  Lemma~\ref{overlap},  we
derive that
\begin{equation*}
    \begin{split}
        \int_X \left(\sum_{n=a}^b (f_n(\mathbf{x}) - f_n)\right)^2 d\mu(\mathbf{x})
        &\ll \left(\sum_{n=a}^b \psi(n)\right)^{\frac{2d}{d+1}} \log^+\left(\sum_{n=a}^b \psi(n)\right) + \sum_{n=a}^b \psi(n) \\
        &= \sum_{n=a}^b \phi(n),
    \end{split}
\end{equation*}
where  \(\phi(n) := \psi(n)\Psi(n)^{\frac{d-1}{d+1}}(\log^+ \Psi(n) + 1) + 2\psi(n)\) . The equality follows from  \cite[Section 3.3]{PVZZ22}.

Applying the summation technique from Lemma~\ref{Philipp}, we obtain
\begin{align*}
    \sum_{n=1}^{N} \phi(n)
    &\leq \sum_{n=1}^{N}\psi(n)\Psi(N)^{\frac{d-1}{d+1}}(\log^{+}\Psi(N)+1) + 2\sum_{n=1}^{N}\psi(n) \\
    &= \Psi(N)^{\frac{2d}{d+1}}(\log^{+}\Psi(N)+1) + 2\Psi(N),
\end{align*}
where \(\Psi(N) = \sum_{n=1}^N \psi(n)\). Consequently, for any \(\varepsilon > 0\), we achieve that
\begin{equation*}
    R(\mathbf{x}, N) = \Psi(N) + O\left(\Psi(N)^{\frac{d}{d+1}} (\log \Psi(N) + 2)^{2+\varepsilon}\right) \quad \mu\text{-a.e. } \mathbf{x} \in E.
\end{equation*}

\medskip
{\noindent \bf  Acknowledgements}. The authors would like to thank Baowei Wang (HUST) for bringing the problem and providing valuable discussions. This work is supported by National Key R$\&$D Program of China (No. 2024YFA1013700) and NSFC (Nos. 12171172, 12201476 and 12331005).

\end{document}